\documentclass[reqno]{amsart}

\usepackage{graphicx}
\usepackage{amsfonts}
\usepackage{amsmath}
\usepackage{amssymb}
\usepackage{amsthm}
\usepackage{enumerate}
\usepackage{esint}
\usepackage{geometry}
\usepackage{musicography}
\usepackage{mathtools}
\usepackage{hyperref}
\usepackage{listings}

\usepackage[]{hyphenat}

\AtBeginDocument{
  
}

\newtheorem{theorem}{Theorem}

\newtheorem{lemma}[theorem]{Lemma}
\newtheorem{remark}[theorem]{Remark}
\usepackage{esint}
\usepackage{xcolor}

\title{Classification of singularities of planar slowness surfaces}
\author{Antonio Cocan}
\author{Maarten V. de Hoop}
\author{Joonas Ilmavirta}
\author{Pieti Kirkkopelto}
\author{Antti Kykkänen}
\date{\today}

\newcommand{\R}{\mathbb{R}}
\newcommand{\C}{\mathbb{C}}

\newcommand{\abs}[1]{\left\lvert#1\right\rvert}

\DeclareMathOperator{\tr}{tr}

\DeclareMathOperator{\Sym}{Sym}
\newcommand{\der}{\mathrm{d}}
\newcommand{\bundleset}{\mathcal{B}}
\newcommand{\baseset}{\mathcal{H}}
\newcommand{\fiberset}{\mathcal{V}}
\newcommand{\singularset}{S}
\newcommand{\sphere}[1]{\mathbb{S}^{#1}}

\newcommand{\V}{\mathbb{V}}
\newcommand{\bolda}{\mathbf{a}}
\newcommand{\boldc}{\mathbf{c}}

\begin{document}

\begin{abstract}
Slowness surfaces are algebraic varieties arising from propagation of elastic waves. In dimensions $2$, we completely classify the types of singularities slowness surfaces can have. The two types of possible singularities are a transversal self-intersection and a tangential singularity produced by a concentric circle and ellipse that are tangent to each other.

To interpret these results analytically, in the case that the slowness surface has transversal self-intersections, we show that the principal symbol of the elastic wave operator is locally smoothly diagonalizable.
\end{abstract}

\maketitle

\section{Introduction}

\subsection{Algebraic results}

A slowness surface $\Sigma_\bolda\subset\R^n$ is an algebraic variety parametrized by a (density-normalized) stiffness tensor $\bolda$ (a rank $4$ tensor with suitable symmetries\footnote{
The required symmetries for the components of $\bolda$ are $a_{ijkl}=a_{jikl}=a_{klij}$.
This can be encoded as $\bolda$ being an element of the symmetric tensor square of the symmetric tensor square of $\R^n$.
See section~\ref{sec:defs} for details.
}):
Any $\bolda\in\Sym^2_+(\Sym^2(\R^n))$ and $\xi\in\R^n$ define the Christoffel matrix $\Gamma_\bolda(\xi)$ whose $ij$:th entry is $\sum_{k,l}a_{iklj}\xi_k\xi_l$, which in turn defines the slowness polynomial $P_\bolda(\xi)=\det(\Gamma_\bolda(\xi)-I)$, whose vanishing set is the slowness surface $\Sigma_\bolda$ (see e.g.~\cite{dHILV2023} for more details).

In this paper, we classify the different kinds of singularities that slowness surfaces can have in dimension $n=2$.
In~\cite[Theorems 2]{IKK2025} a full characterization was given of stiffness tensors $\bolda$ for which $\Sigma_\bolda$ is singular, that is, the question as to \textit{when} a singularity occurs was answered. We now focus on understanding \textit{what} kinds of singularities are possible and when each one happens.

While algebraic curves in the plane can exhibit many kinds of singular behavior, slowness surfaces can only have two kinds of singularities:
either a transversal crossing, which we call an ordinary node (cf.~\cite{CA2000}) or an ordinary tacnode (cf.~\cite{CA2000}), given by a concentric circle and ellipse being tangent to each other, as we make precise in the following theorem.

\begin{theorem}
\label{classification}
Let $\bolda\in\Sym^2_+(\Sym^2(\R^2))$. The slowness surface $\Sigma_\bolda$ can have exactly two types of singularities: ordinary nodes (transversal self-intersection) or ordinary tacnodes (a concentric circle and an ellipse tangent to each other).
\end{theorem}

In applications the stiffness tensor is real, and one is primarily interested in the real points of the slowness surface. From the algebraic point of view, however, it is also natural to consider complex stiffness tensors and complex points. Our first main result gives an algebraic characterization of the tangential, or ordinary tacnode, singularities.

\begin{theorem}
\label{thm:classification-v2}
Let $\bolda\in\Sym^2_+(\Sym^2(\R^2))$. The following are equivalent.
\begin{enumerate}
    \item $\Sigma_\bolda$ has an ordinary tacnode singularity.
    \item $\Sigma_\bolda$ is a union of a circle and an ellipse that are tangent at two opposite points.
    \item The following six polynomials vanish:
\begin{equation}
\label{eq:ordinary tacnode_sing_polys-full}
\begin{split}
Q_1(\bolda) &= 
a_{1122}a_{2222}
- a_{1122}a_{1212}
+ 2a_{1112}a_{2212}
+ a_{2222}a_{1212}
- 2a_{2212}^2
- a_{1212}^2,
\\
Q_2(\bolda) &= a_{1122}^2
+ 2a_{1122}a_{1212}
- 4a_{1112}a_{2212}
+ a_{1212}^2,
\\
Q_3(\bolda) &= a_{1111}a_{2212}
+ a_{1112}a_{2222}
- a_{1112}a_{1212}
- a_{2212}a_{1212},
\\
Q_4(\bolda) &= 2a_{1111}a_{2212}
- a_{1122}a_{1112}
+ a_{1122}a_{2212}
- a_{1112}a_{1212}
- a_{2212}a_{1212},
\\
Q_5(\bolda) &= a_{1111}a_{2222}
- a_{1111}a_{1212}
+ a_{1112}^2
- 2a_{1112}a_{2212}
- a_{2222}a_{1212}
+ a_{2212}^2
+ a_{1212}^2,
\\
Q_6(\bolda) &= a_{1111}a_{1122}
+ 2a_{1111}a_{2222}
- a_{1111}a_{1212}
- a_{1122}a_{1212}
- 2a_{1112}a_{2212} - 2a_{2222}a_{1212}
+ 2a_{2212}^2
+ a_{1212}^2.
\end{split}
\end{equation}
\end{enumerate}
The dimension of the set of such stiffness tensors $\bolda$ is $3$.
\end{theorem}

It turns out that real singularities are rarer than complex ones in a sense we quantify in the following theorem.

\begin{theorem}
\label{thm:main2}
Let $A\cong\C^6$ be the space of all complex stiffness tensors.
Let $S\subset A$ be the subset whose slowness surfaces have a singular point, and $U\subset S$ the subset where there is an ordinary tacnode singular point.
Then $S$ and $U$ are Zariski-closed and their dimensions are 5 and 4, respectively.
\end{theorem}

\subsection{Analytic results}

An elastic material is described by a density function $\rho>0$ and a stiffness tensor~$\boldc\in\Sym^2_+(\Sym^2(\R^n))$. These material parameters may vary with position, and small perturbations $u(t,x)\in\R^n$ from equilibrium satisfy the linear elastic wave equation
\begin{equation}
\rho(x)\partial_t^2u_i(t,x)
-\sum_{j,k,l}\partial_{x_j}[c_{ijkl}(x)\partial_{x_k}u_l(t,x)]
=
0.
\end{equation}
The full solution theory of this system is subtle, but the propagation of singularities is governed by the principal symbol of the spatial elastic operator, namely the Christoffel matrix.

When the relevant characteristic roots are smooth, or more generally when the system is of real principal type in the sense of~\cite{Dencker1982}, singularities propagate along the Hamiltonian flows associated with the eigenvalues of the principal symbol. The union of the corresponding level sets is precisely the slowness surface for the density-normalized stiffness tensor $\bolda=\rho^{-1}\boldc$. Thus, at the level of geometric optics, slowness surfaces encode the geometry along which elastic waves propagate.

At singular points of the slowness surface this description requires additional care. The ordinary node singularities isolated above are still microlocally tractable. If $\bolda(x)$ varies smoothly with $x\in\Omega\subset\R^n$, then each cotangent fiber $T^*_x\Omega$ carries a slowness surface $\Sigma_{\bolda(x)}$. The next theorem gives the local smooth diagonalization result that underlies the microlocal analysis near ordinary node singularities.

\begin{theorem}
\label{thm:ordinary node->diagonalizable}
Let $\Omega\subset\R^n$ be open. Suppose that $\boldc\colon\Omega\to\Sym^2_+(\Sym^2(\R^n))$ and $\rho\colon\Omega\to(0,\infty)$ are $C^\infty$, and set $\bolda=\rho^{-1}\boldc$. Let $x_0\in\Omega$ and $\xi_0\in T^*_{x_0}\Omega\setminus\{0\}$. Assume that $(x_0,\xi_0)$ has a neighborhood $\bundleset\subset T^*\Omega$ in which all singularities of the fiberwise slowness surfaces are ordinary node. Then, after possibly shrinking to a neighborhood $\bundleset'\subset\bundleset$ of $(x_0,\xi_0)$, there is a smooth $0$-homogeneous map $U\colon\bundleset'\to O(n)$ such that
\begin{equation}
U(x,\xi)\Gamma_{\bolda(x)}(\xi)U(x,\xi)^T
\end{equation}
is diagonal for every $(x,\xi)\in\bundleset'$.
\end{theorem}

For background on propagation of polarized waves, we refer the reader to~\cite{Hintz2025}.

\subsection{Related results}

The geometry of slowness surfaces and the behavior of elastic wave polarizations have long been studied in anisotropic elasticity. Early work includes that of Yedlin~\cite{Yedlin1980} on wave fronts in homogeneous anisotropic media, Helbig and Carcione~\cite{HelbigCarcione2009} on anomalous polarization phenomena, and Mazzucato and Rachele~\cite{MazzucatoRachele2008}, who investigated transversely isotropic elastic media with ellipsoidal slowness surfaces. These works highlight the intimate relationship between the geometry of the slowness surface, polarization, and wave propagation.

More recently, algebraic geometric methods have proved useful in the study of slowness surfaces and elastic inverse problems. In \cite{dHILV2023}, techniques from algebraic geometry are used to show that, for a generic anisotropic stiffness tensor, local information around a single polarization uniquely determines the entire slowness surface and, consequently, the underlying stiffness tensor. The reconstruction is made effective through Gröbner basis computations, illustrating the usefulness of computational algebraic geometry in anisotropic elasticity.

The present work builds in particular on~\cite{IKK2025}, where the occurrence of singularities for two-dimensional slowness surfaces was characterized. In particular, it was shown that the singular points of the slowness surface coincide precisely with the degeneracy points of the Christoffel matrix~\cite[Proposition~8]{IKK2025}. More precisely, it was shown that there are explicit polynomials $R$ and $D$ in the stiffness tensor components $a_{ijkl}$ so that $\Sigma_\bolda$ has a singularity if and only if $R(\bolda) = 0$ and $D(\bolda) > 0$.

%
%

The present paper strengthens this characterization by determining the type of singularities that occur. More precisely, Theorem~\ref{classification} shows that every singularity is either of ordinary node or ordinary tacnode, while Theorem~\ref{thm:classification-v2} gives a purely algebraic classification of the tangential (or ordinary tacnode) singularities using the polynomial conditions~\eqref{eq:ordinary tacnode_sing_polys-full}. Thus, the algebraic characterization of the singular locus obtained in \cite{IKK2025} is refined into a complete algebraic classification of the local singularity type. Our work does not follow this line of work directly, since the involutivity condition assumed in~\cite{Dencker1995} is not necessarily satisfied in our setting (see Appendix~\ref{app:involutive-fails}). 

On the analytic side, our diagonalization results are related to the microlocal analysis of hyperbolic systems and the propagation of polarization sets. Dencker~\cite{Dencker1982,Dencker1988,Dencker1995} established the propagation of the polarization set for various types of systems. Also, Braam and Duistermaat~\cite{BraamDuistermaat1993}, Colin de Verdi\`ere~\cite{ColinDeVerdiere2003,ColinDeVerdiere2004}, Melrose and Uhlmann~\cite{MelroseUhlmann1979,Uhlmann1982} studied related phenomena arising from eigenvalue multiplicities and conical refraction. More recently, \cite{OSSU2024} established diagonalization results for systems of real principal type under geometric assumptions on the characteristic variety, and related propagation results were obtained in \cite{darwich2023propagation}. 

Beyond the local analysis considered here, anisotropic elasticity has motivated a broad range of inverse problems. We refer, for example, to the inverse boundary value problems of \cite{NakamuraUhlmann1994,EskinRalston2002,ImanuvilovUhlmannYamamoto2012,BerettaFranciniVessella2014,BarceloEtAl2018}, to inverse problems for anisotropic and transversely isotropic media \cite{MazzucatoRachele2006,MazzucatoRachele2007,NakamuraTanumaUhlmann1999,SacksYakhno1998,DeHoopUhlmannVasy2020,Zou2021}, and to recent developments connecting the geometry of the qP slowness surface with Finsler geometry and geometric inverse problems \cite{DeHoopIlmavirtaLassasSaksala2021,DeHoopIlmavirtaLassasSaksala2023,DeHoopIlmavirtaLassas2025,IlmavirtaMonkkonen2022,PaternainSaloUhlmann2023,UhlmannVasy2016}.

\subsection{Definitions}
\label{sec:defs}
For a finite-dimensional real inner product space $E$, we denote by $\Sym^2(E)$ the symmetric tensor square, identified with the space of quadratic forms on $E$. The cone $\Sym^2_+(E)\subset\Sym^2(E)$ consists of the positive definite quadratic forms.

A stiffness tensor $\bolda\in\Sym^2_+(\Sym^2(\R^n))$ has components satisfying the elastic symmetries
$a_{ijkl}=a_{jikl}=a_{klij}$.
The positivity condition is
$\sum_{i,j,k,l}a_{ijkl}A_{ij}A_{kl}>0$ pointwise for every non-zero symmetric matrix $A$.

In dimension $n=2$, the slowness polynomial is explicitly
\begin{equation}
\label{eq:general_slowness_polynomial_in_dimension_two}
\begin{aligned}
P_{\bolda}(x, y)
&=
\det(\Gamma_{\bolda}(x, y)-I)
\\&=
(a_{1111}a_{1212} - a_{1112}^{2})x^{4}
+ (2a_{1111}a_{2212} - 2a_{1112}a_{1122})x^{3}y
\\&
\quad
+ (a_{1111}a_{2222} + 2a_{1112}a_{2212} - a_{1122}^{2} - 2a_{1122}a_{1212})x^{2}y^{2}
- (a_{1111} + a_{1212})x^{2}
\\&
\quad
+ (2a_{1112}a_{2222} - 2a_{2212}a_{1122})xy^{3}
- (2a_{1112} + 2a_{2212})xy
\\&
\quad
+ (a_{1212}a_{2222} - a_{2212}^{2})y^{4}
- (a_{1212} + a_{2222})y^{2}
+ 1
\in k[x,y].
\end{aligned}
\end{equation}
Its vanishing set $\Sigma_\bolda=\V(P_\bolda)$ is the planar slowness surface.

\subsection{Organization of the article}

The paper is organized as follows. The remainder of the introduction fixes notation and recalls the basic definitions. Section~\ref{sec:algebraic_proofs} contains the algebraic proofs of our main Theorems~\ref{classification},~\ref{thm:classification-v2} and~\ref{thm:main2}; as well as proofs of the necessary lemmas. Section~\ref{sec:analytic_proofs} proves the Analytic Diagonalization Theorem~\ref{thm:ordinary node->diagonalizable} and the needed lemmas. The appendices include a short discussion on how the elastic wave operator with fiberwise ordinary nodes on slowness surfaces relates to systems of transverse type as defined in ~\cite{Dencker1995} and a record of the computer algebra calculations used in the algebraic arguments. We keep the background on the elastic wave operator short as to not repeat ourselves since these have been covered already in other articles such as~\cite{dHILV2023} and~\cite{IKK2025}.

\subsection{Acknowledgements}

Antonio Cocan and Joonas Ilmavirta were supported by the Research Council of Finland (Flagship of Advanced Mathematics for Sensing Imaging and Modelling grant 359208; Centre of Excellence of Inverse Modelling and Imaging grant 353092; and other grants 351665, 351656, 358047, 360434) and a Väisälä project grant by the Finnish Academy of Science and Letters. 
Maarten V. de Hoop was supported by the Simons Foundation under the MATH + X program, the National Science Foundation under grant DMS-1815143, and the corporate members of the Geo-Mathematical Imaging Group at Rice University.
Pieti Kirkkopelto was supported by the Finnish Ministry of Education and Culture’s Pilot for Doctoral Programmes (Pilot project Mathematics of Sensing, Imaging and Modeling) and the Research Council of Finland
(Flagship of Advanced Mathematics for Sensing Imaging and Modelling grant 359208; Centre of Excellence of Inverse Modelling and Imaging grant 353092).
Antti Kykkänen was supported by the Geo-Mathematical Imaging Group at Rice University.

\section{Algebraic proofs}
\label{sec:algebraic_proofs}

In this section we assume basic knowledge of algebraic geometry.
For an introduction, we recommend~\cite{CLO2025,SK2000}.

\subsection{Lemmas}

In this section we gather the necessary lemmas needed for the proofs of our algebro-geometric Theorems~\ref{classification},~\ref{thm:classification-v2} and~\ref{thm:main2}. Let $f_\bolda(\theta)$ be the discriminant of the characteristic polynomial of $\Gamma_\bolda(p)$ at $p=(\cos\theta,\sin\theta)$.

\begin{lemma}
\label{lma:f=0}
There is a singularity on the slowness surface $\Sigma_\bolda$ in the direction $\theta$ if and only if $f_\bolda(\theta)=0$. This implication does not require positive definiteness.
\end{lemma}

\begin{lemma}
\label{lma:f>0}
For every positive definite $\bolda$, the function $f_\bolda\colon S^1\to\R$ is smooth and non-negative. In particular, $f_\bolda(\theta)=0$ implies $f'_\bolda(\theta)=0$.
\end{lemma}

\begin{lemma}
\label{lma:f''=0}
Let $\bolda$ be a positive definite stiffness tensor.
The conditions $f_\bolda(0)=f'_\bolda(0)=f''_\bolda(0)=0$ hold if and only if
\begin{equation}
\label{eq:ellipse-tensor}
\begin{cases}
a_{11}=r \\
a_{12}=-r \\
a_{13}=0 \\
a_{22}=t \\
a_{23}=0 \\
a_{33}=r
\end{cases}
\end{equation}
for some $t>r>0$.
\end{lemma}

\begin{lemma}
\label{lma:f''>0}
Let $\bolda$ be a positive definite stiffness tensor.
If $f_\bolda(0)=f'_\bolda(0)=0$ and $f''_\bolda(0)\neq0$, then the singularity in direction $\theta=0$ is an ordinary node.
\end{lemma}

\begin{lemma}
\label{lem:complex_circleellipse_no_real}
Let $(r,a,b,c)\in\C^4$ and $(a_{11},\ldots,a_{33})\in\R^6$ satisfy the coefficient comparison and tangency system~\eqref{eq:tangency_generators_system}. Then $r,a,b,c\in\R$. In particular, complex circle--ellipse parameters do not produce additional real stiffness tensors in this system.
\end{lemma}

\subsection{Proofs of algebro-geometric theorems}

We now proceed to prove our algebro-geometric Theorems~\ref{classification},~\ref{thm:classification-v2} and~\ref{thm:main2}.

\begin{proof}[Proof of Theorem \ref{classification}]
Let $\theta$ correspond to a singular point of the slowness polynomial $P_\bolda$. If the Hessian of $P_\bolda$ at this point is non-degenerate, Lemma~\ref{lma:f''>0} shows that the singularity is an ordinary node. If the Hessian is degenerate, Theorem~\ref{thm:classification-v2} shows that the singularity is tangential.
\end{proof}

\begin{proof}[Proof of Theorem~\ref{thm:classification-v2}]

$(1) \implies (2)$: Suppose first that the ordinary tacnode singularity occurs in the direction $\theta=0$; after a rotation of the coordinates, this entails no loss of generality.
By Lemmas~\ref{lma:f=0} and~\ref{lma:f>0} we have that $f_\bolda(0)=f'_\bolda(0)=0$.
By Lemma~\ref{lma:f''>0} untidiness implies that $f''_\bolda(0)=0$, and thus by Lemma~\ref{lma:f''=0} the stiffness tensor is of the form~\eqref{eq:ellipse-tensor}.
In this case the slowness polynomial is
\begin{equation}
\label{eq:ellipse-polynomial}
P_\bolda(p)
=
(rp_1^2+rp_2^2-1)
(rp_1^2+tp_2^2-1)
\end{equation}
so the slowness surface is the union of a circle and an ellipse tangent at $(\pm r^{-1/2},0)$.

$(2) \implies (1)$: Conversely, a tangency point of such a circle--ellipse union cannot be a non-degenerate critical point of the slowness polynomial. Otherwise Lemma~\ref{lma:f''>0} would give a crossing-type singularity, contradicting tangency.

$(2) \iff (3)$ This equivalence is obtained by elimination using \textsc{Magma}, which can be found in \texttt{algebraic\_characterization.magma}. Consider affine real space $\mathbb{A}^{6}_{\bolda}$ with coordinate ring $A = \mathbb{R}[\bolda]$ and affine real space $\mathbb{A}^{4}_{r, a, b, c}$ with coordinate ring $B = \mathbb{R}[r, a, b, c]$. A circle polynomial and an ellipse polynomial are given respectively by
\begin{equation}
    \begin{aligned}
        C &:= rx^2+ry^2-1, \\
        E &:= ax^2+2cxy + by^2-1.
    \end{aligned}
\end{equation}
Their union is described by the product $C\cdot E$. Given a general slowness polynomial $P_\bolda$ described by~\eqref{eq:general_slowness_polynomial_in_dimension_two} we compare coefficients to get an ideal
\begin{equation}
    G = \langle P_\bolda - C \cdot E \rangle.
\end{equation}
The ideal $G$ records the coefficient conditions for the factorization $P_\bolda=C E$. The tangency condition between an ellipse and a circle is encoded using the ideal
\begin{equation}
    T = \langle (a-r)\cdot(b-r)-c^2 \rangle.
\end{equation}
Since both curves are centered at the origin, tangency occurs precisely when the two associated quadratic forms share an eigenvalue, which gives the condition defining $T$. Then the ideal
\begin{equation}
    I = G + T = \langle P_\bolda - C \cdot E, (a-r)\cdot(b-r)-c^2 \rangle
\end{equation}
cuts out the slowness polynomials whose zero sets are unions of a circle and an ellipse tangent to one another. Eliminating $r,a,b,c$ gives the generators $Q_{i}(\mathbf{a})$ as listed in~\eqref{eq:ordinary tacnode_sing_polys-full}. These are precisely the polynomial relations among the $a_{ijkl}$ that are necessary and sufficient for the slowness polynomial to factor as a circle times an ellipse with tangency.
\end{proof}

\begin{proof}[Proof of Theorem~\ref{thm:main2}]
The \textsc{Magma} code needed for this proof is found in the supporting file \texttt{complexdimension\_singularities.magma}. We work over $\C$ in the polynomial ring
$
    R = \C[\bolda][x, y]
$. The slowness polynomial $P_{\bolda}$ as in~\eqref{eq:general_slowness_polynomial_in_dimension_two} and its partial derivatives $\partial_1 P_{\bolda}, \partial_2 P_{\bolda}$ generate the ideal
\begin{equation}
    I = \langle P_{\bolda}, \partial_1 P_{\bolda}, \partial_2 P_{\bolda} \rangle
    \subset R
    .
\end{equation}
Its vanishing set
$
    \V(I) \subset 
    \mathbb{A}^{6}_{\bolda}
    \times \mathbb{A}^{2}_{x, y}
$
is the incidence variety of singular points.

Eliminating the variables $x$ and $y$ gives the corresponding elimination ideal $eI\subset \C[\bolda]$,
which describes those stiffness tensors that have singular points in their slowness surfaces.
Let $S$ be the set of stiffness tensors $\bolda$ for which $\Sigma_\bolda$ has a singular point.
By elimination theory, the Zariski closure of $S$ is $\V(eI)$, and $S$ and $\V(eI)$ have the same dimension. A \textsc{Magma} computation gives Krull dimension $5$.

If the Hessian is invertible, then by the holomorphic Morse lemma the complex slowness surface is locally a union of two smooth curves.
To characterize the locus where the singularity is not ordinary node, we replace $I$ by
\begin{equation}
    J = \langle P_{\bolda}, \partial_1 P_{\bolda}, \partial_2 P_{\bolda}, H P_{\bolda} \rangle
    \subset R,
\end{equation}
where $H P_{\bolda}$ denotes the determinant of the Hessian of $P_\bolda$.
Repeating the same elimination and dimension computation shows that the set of stiffness tensors with ordinary tacnode singularities has dimension $4$.
\end{proof}

\subsection{Proofs of algebro-geometric lemmas}

\begin{proof}[Proof of Lemma~\ref{lma:f=0}]
    In the real case the geometric multiplicity of eigenvalues equals the algebraic multiplicity of the roots of the characteristic polynomial. The characteristic polynomial has a double root if and only if $f=0$.
\end{proof}

\begin{proof}[Proof of Lemma~\ref{lma:f>0}]
    By the spectral theorem the characteristic polynomial always has real roots. Hence $f\ge 0$.  Smoothness follows by observing that $f$ is a composition of smooth functions:
    \begin{equation}
        f=\mathrm{Tr}(\Gamma)^2(p(\theta))-4\mathrm{det}(\Gamma)(p(\theta)),
    \end{equation}
    where the function 
    \begin{equation}
        p:\theta\mapsto (\cos (\theta),\sin (\theta))
    \end{equation}
    is smooth. Combining gives $f(\theta)=0\implies f'(\theta)=0$.
\end{proof}

\begin{proof}[Proof of Lemma~\ref{lma:f''=0}]
Assume first that $f_\bolda(0) = f'_\bolda(0) = f''_\bolda(0) = 0$.
Since the polynomial conditions are very long, we use \textsc{Maxima} to find when these conditions hold; see Appendix~\ref{app:maxima}. The output of the \textsc{Maxima} computation yields:

\begin{enumerate}
\label{eq:conds-for-f-f-f-to-vanish}
\item
$a_{11}=r,a_{12}=s,a_{13}=t,a_{22}=-2s-r,a_{23}=-t+is+ir,a_{33}=2it+r$
\item
$a_{11}=r,a_{12}=s,a_{13}=t,a_{22}=-2s-r,a_{23}=-t-is-ir,a_{33}=r-2it$
\item
$a_{11}=r,a_{12}=s,a_{13}=0,a_{22}=t,a_{23}=is+ir,a_{33}=r$
\item
$a_{11}=r,a_{12}=s,a_{13}=0,a_{22}=t,a_{23}=-is-ir,a_{33}=r$
\item
$a_{11}=r,a_{12}=s,a_{13}=-is-ir,a_{22}=-2s-r,a_{23}=0,a_{33}=-2s-r$
\item
$a_{11}=r,a_{12}=s,a_{13}=is+ir,a_{22}=-2s-r,a_{23}=0,a_{33}=-2s-r$
\item
$a_{11}=r,a_{12}=-r,a_{13}=s,a_{22}=r,a_{23}=-s,a_{33}=r-2is$
\item
$a_{11}=r,a_{12}=-r,a_{13}=s,a_{22}=r,a_{23}=-s,a_{33}=r+2is$
\item
$a_{11}=r,a_{12}=-r,a_{13}=0,a_{22}=r,a_{23}=0,a_{33}=r$.
\end{enumerate}
Since the stiffness tensor is real, the cases above reduce to
\begin{enumerate}
\item
$a_{11}=r,a_{12}=-r,a_{13}=0,a_{22}=r,a_{23}=0,a_{33}=r$
\item
$a_{11}=r,a_{12}=-r,a_{13}=0,a_{22}=r,a_{23}=0,a_{33}=r$
\item
$a_{11}=r,a_{12}=-r,a_{13}=0,a_{22}=t,a_{23}=0,a_{33}=r$
\item
$a_{11}=r,a_{12}=-r,a_{13}=0,a_{22}=t,a_{23}=0,a_{33}=r$
\item
$a_{11}=r,a_{12}=-r,a_{13}=0,a_{22}=r,a_{23}=0,a_{33}=r$
\item
$a_{11}=r,a_{12}=-r,a_{13}=0,a_{22}=r,a_{23}=0,a_{33}=r$
\item
$a_{11}=r,a_{12}=-r,a_{13}=0,a_{22}=r,a_{23}=0,a_{33}=r$
\item
$a_{11}=r,a_{12}=-r,a_{13}=0,a_{22}=r,a_{23}=0,a_{33}=r$
\item
$a_{11}=r,a_{12}=-r,a_{13}=0,a_{22}=r,a_{23}=0,a_{33}=r$.
\end{enumerate}
The only conditions in the list above producing a positive definite stiffness tensor are items (3) and (4) in the case that $t > r$. Hence~\eqref{eq:ellipse-tensor} must hold.

Then suppose that the stiffness tensor is given by~\eqref{eq:ellipse-tensor}. Then the stiffness tensor is type (3) for $s = -r$ in~\eqref{eq:conds-for-f-f-f-to-vanish}. Hence $a$ satisfies $f_\bolda(0) = f_\bolda'(0) = f''_\bolda(0) = 0$, proving the equivalence.
\end{proof}

\begin{proof}[Proof of Lemma~\ref{lma:f''>0}]
    Suppose $f''(0)\neq 0$. The determinant of the Hessian of the slowness polynomial at $\theta=0$ is given by $-2f''(0)$. Since $f\ge 0$ and $f(0)=f'(0)=0$, we must have $f''(0)>0$. Hence the determinant is negative, giving us in two dimensions that the Hessian has one negative and one positive eigenvalue. By the Morse lemma this implies that there is a local diffeomorphism $\phi$ fixing $0$ such that
    \begin{equation}
        (P_c\circ\phi^{-1})(z)=z_1^2-z_2^2.\qedhere
    \end{equation}
\end{proof}

\begin{proof}[Proof of Lemma~\ref{lem:complex_circleellipse_no_real}]

The system obtained by comparing the coefficients of a slowness polynomial with the product of a centered circle and a centered ellipse, together with the circle--ellipse tangency condition, is
\begin{equation}
\label{eq:tangency_generators_system}
\begin{aligned}
- r a + a_{11} a_{33} - a_{13}^2 &= 0, \\
- 2 r c + 2 a_{11} a_{23} - 2 a_{12} a_{13} &= 0, \\
- r a - r b + a_{11} a_{22} - a_{12}^2 - 2 a_{12} a_{33} + 2 a_{13} a_{23} &= 0, \\
r + a - a_{11} - a_{33} &= 0, \\
- 2 r c - 2 a_{12} a_{23} + 2 a_{13} a_{22} &= 0, \\
2 c - 2 a_{13} - 2 a_{23} &= 0, \\
- r b + a_{22} a_{33} - a_{23}^2 &= 0, \\
r + b - a_{22} - a_{33} &= 0, \\
r^2 - r a - r b + a b - c^2 &= 0.
\end{aligned}
\end{equation}
Assume that $(r,a,b,c,a_{11},\ldots,a_{33})\in\C^4\times\R^6$ solves~\eqref{eq:tangency_generators_system}. We show that the parameters $r,a,b,c$ are necessarily real. The equation
\begin{equation}
    2 c - 2 a_{13} - 2 a_{23} = 0 \iff c = a_{13} + a_{23}
\end{equation}
shows that $c$ is real. If $c\neq0$, then the equation
\begin{equation}
    - 2 r c - 2 a_{12} a_{23} + 2 a_{13} a_{22} = 0 \iff r = \frac{a_{13} a_{22} - a_{12} a_{23}}{c}
\end{equation}
implies that $r$ is real. Since $r$ is real, the equation
\begin{equation}
    r + a - a_{11} - a_{33} = 0 \iff a = a_{11} + a_{33} - r
\end{equation}
implies that $a$ is real. Similarly,
\begin{equation}
    r + b - a_{22} - a_{33} = 0 \iff b = a_{22} + a_{33} - r
\end{equation}
implies that $b$ is real. This proves the claim when $c\neq0$. It remains to consider the case $c=0$. Then~\eqref{eq:tangency_generators_system} reduces to
\begin{equation}
\label{eq:tangency_generators_system_reduced}
\begin{aligned}
- r a + a_{11} a_{33} - a_{13}^2 &= 0, \\
a_{11} a_{23} - a_{12} a_{13} &= 0, \\
- r a - r b + a_{11} a_{22} - a_{12}^2 - 2 a_{12} a_{33} + 2 a_{13} a_{23} &= 0, \\
r + a - a_{11} - a_{33} &= 0, \\
-a_{12} a_{23} + a_{13} a_{22} &= 0, \\
a_{13} + a_{23} &= 0, \\
- r b + a_{22} a_{33} - a_{23}^2 &= 0, \\
r + b - a_{22} - a_{33} &= 0, \\
r^2 - r a - r b + a b &= 0.
\end{aligned}
\end{equation}
The last equation factors as
\begin{equation}
    r^2 - r a - r b + a b = 0 \iff (r - a)(r - b) = 0,
\end{equation}
Thus either $r=a$ or $r=b$. Suppose first that $r=a$. Then
\begin{equation}
    r + a - a_{11} - a_{33} = 0 \iff 2r = a_{11} + a_{33} \iff r = \frac{a_{11} + a_{33}}{2}
\end{equation}
shows that $r$ is real, and hence $a=r$ is real. Moreover,
\begin{equation}
    r + b - a_{22} - a_{33} = 0 \iff b = a_{22} + a_{33} - r
\end{equation}
shows that $b$ is real. Therefore $r,a,b$ are real and $c=0$. The case $r=b$ is analogous. Thus every solution in $\C^4\times\R^6$ has $r,a,b,c\in\R$, as claimed.
\end{proof}

\section{Analytic proofs}
\label{sec:analytic_proofs}

We move on to the analytic implications of transverse type singularities of slowness surfaces. 

\subsection{Lemmas}

Our proof of the microlocal diagonalizability result relies on the following set of lemmas, proven at the end of this section.

In all the lemmas of this subsection we assume that $\rho(x)$ and $\boldc(x)$ are $C^\infty$ and they satisfy the positivity and symmetry conditions and that the slowness surfaces are ordinary node in an open set $\bundleset\subset T^*\Omega\setminus0$.
As the final claims are all local, we assume for convenience that $\bundleset=\baseset\times\fiberset\subset\R^2\times\R^2$ with convex open $\baseset,\fiberset\subset\R^2$ without loss of generality for the theorems stated above.

Let us denote by $\singularset\subset\bundleset$ the set where $\Gamma$ has a double eigenvalue.
For any function $f$ defined on $\bundleset$ we call the functions $\xi\mapsto f(x,\xi)$ for fixed $x$ the vertical slices of $f$.
The domain is $\bundleset\cap T^*_x\Omega$.

\begin{lemma}
\label{lma:smooth-singular-locus-on-bundle}
The set $\singularset\subset T^*\Omega$ is a smooth hypersurface and there is a smooth non-vanishing one-form $\alpha(x)$ defined on $\baseset$ so that
$
\singularset
=
\{
(x,\xi)
\in\bundleset
;
\xi
\in
\R\alpha(x)
\}
$
.
\end{lemma}

\begin{lemma}
\label{lma:global-existence+vertical-analyticity}
For $i=1,2$ there are functions
$\lambda_i
\colon\bundleset\to(0,\infty)$
and
$q_i
\colon\bundleset\to \sphere{1}\subset\R^2$
so that
(a)
$\Gamma_{\boldc(x)}(\xi)q_i(x,\xi)=\lambda_i(x,\xi)q_i(x,\xi)$,
(b)
$\lambda_1\neq\lambda_2$ on $\bundleset\setminus\singularset$,
and
(c)
the vertical slices of $\lambda_i$ and $q_i$ are real analytic.
\end{lemma}

\begin{lemma}
\label{lma:non-degenerate-smoothness}
The functions $\lambda_i$ and $q_i$ are $C^\infty$ on $\bundleset\setminus\singularset$.
\end{lemma}

\begin{lemma}
\label{lma:Lipschitz}
The eigenvalues $\lambda_i$ are Lipschitz on $\bundleset$.
\end{lemma}

\begin{lemma}
\label{lma:ordinary node-fiber}
For $(x,\xi)\in\singularset$ the $\xi$-differential $\der_\xi(\lambda_1-\lambda_2)$ is non-zero.
\end{lemma}

\begin{lemma}
\label{lma:Hadamard}
Let $U\subset\R^n$, $n\geq2$, be an open convex set.
Fix any integer $k\geq1$.
If $f\in C^\infty(U)$ satisfies
$
\abs{f(x)}
\leq
C\abs{x_1}^k
$
for some $C>0$, then there is a smooth function $H$ so that
$
f(x)
=
x_1^kH(x)
$
.
\end{lemma}

\subsection{Proof of the analytic theorem}

\begin{proof}[Proof of Theorem~\ref{thm:ordinary node->diagonalizable}]
Denote the sum and difference of the eigenvalues by $\mu_\pm=\lambda_1\pm\lambda_2$.
Because $\Gamma(x,\xi)$ is smooth, so is its trace $\mu_+(x,\xi)$.
The squared difference $\mu_-^2=\tr(\Gamma)^2-4\det(\Gamma)$ is smooth, too.
We need to establish that $\mu_-$ is smooth near any point $(x_0,\xi_0)\in\singularset$; elsewhere it is given by Lemma~\ref{lma:non-degenerate-smoothness}.

As we are on the degenerate set, Lemma~\ref{lma:smooth-singular-locus-on-bundle} gives $\xi_0\parallel\alpha(x_0)$.
We can multiply the one-form $\alpha$ so that $\xi_0=\alpha(x_0)$.

Fix any $\beta\in\R^2$ that is linearly independent from $\alpha(x_0)$.
This ensures that $\Phi(x,a,b)=(x,a\alpha(x)+b\beta)\in\bundleset$ is a smooth local diffeomorphism near $\Phi(x_0,1,0)=(x_0,\xi_0)$.
In these new coordinates $\singularset=\{b=0\}$.
We will prove that the function $\mu_-(x,a,b)$ expressed in these new coordinates is smooth near $(x_0,1,0)$.

As $\mu_-$ vanishes on $\singularset$ and is $L$-Lipcshitz for some $L>0$ by Lemma~\ref{lma:Lipschitz}, we get
$\abs{\mu_-(x,a,b)^2}\leq L^2b^2$.
Using Lemma~\ref{lma:Hadamard} on the smooth function $\mu_-^2$ gives a smooth function $H(x,a,b)$ for which $\mu_-(x,a,b)^2=b^2H(x,a,b)$.

Next we use this smooth $H$ and the good vertical properties of the eigenvalues to ensure that $\mu_-$ is smooth on the bundle; smoothness along each fiber is provided by Lemma~\ref{lma:global-existence+vertical-analyticity}.
Since $\mu_-=0$ along $b=0$ and the vertical differential is non-zero by Lemma~\ref{lma:ordinary node-fiber}, we must have $\partial_b\mu_-(x,a,0)\neq0$.
Therefore also
\begin{equation}
H(x,a,0)
=
[\partial_b\mu_-(x,a,0)]^2
>
0
.
\end{equation}
Upon shrinking $\bundleset$ if necessary, the function $H$ is positive and has a smooth positive square root $h=\sqrt{H}$.
We have thus found that
\begin{equation}
\label{eq:mu-minus-h}
\mu_-(x,a,b)
=
s(x,a,b)
b
h(x,a,b)
,
\end{equation}
where $s(x,a,b)=\pm1$ is a sign.

By smoothness of $\mu_-$ outside $\singularset$, the sign $s$ is locally constant in $\bundleset\setminus\singularset$.
The vertical slices of $s$ are smooth by Lemma~\ref{lma:global-existence+vertical-analyticity}, so $s$ is in fact a global constant and $\mu_-\in C^\infty(\bundleset)$.

We have now established that the eigenvalues $\frac12(\mu_+\pm\mu_-)$ are smooth in $\bundleset$ as claimed.
We then turn to eigenvectors $q_i$ of $\Gamma$, continuing to use the same coordinates $(x,a,b)$.

Let $P_i=q_iq_i^T$ be the orthogonal projector to the subspace $q_i\R$.
When $b\neq0$, the eigenvalue is simple and we can write $P_1$ as
\begin{equation}
P_1(x,a,b)
=
\frac{\Gamma_{c(x)}(\xi)-\lambda_2(x,\xi)I}{\lambda_1(x,a,b)-\lambda_2(x,a,b)}
.
\end{equation}
The matrix elements of the matrix in the numerator vanish at $b=0$, so by applying Lemma~\ref{lma:Hadamard} to each element we see that it can be written as $bM(x,a,b)$ for a smooth matrix-valued function $M$.
This together with equation~\eqref{eq:mu-minus-h} gives
\begin{equation}
P_1(x,a,b)
=
\frac{M(x,a,b)}{sh(x,a,b)}
.
\end{equation}
Both $M$ and $h$ are smooth and $h>0$, so $P_1$ is smooth on all of $\bundleset$.

Because $\tr(P_1)=1$ everywhere, $P_1(x,a,b)$ must, upon potentially shrinking the domain, have a positive diagonal entry.
Suppose without loss of generality that $P_1(x,a,b)_{11}>0$ at all points.

We can express the first component of the eigenvector $q_1(x,a,0)$ as
\begin{equation}
q_1(x,a,b)_1
=
\hat s(x,a,b)\sqrt{P_1(x,a,b)_{11}}
\end{equation}
with a variable sign $\hat s(x,a,b)=\pm1$.
By Lemma~\ref{lma:global-existence+vertical-analyticity} the vertical slices of eigenvectors are smooth, so $\hat s(x,a,b)$ is independent of $(a,b)$.
For $b\neq0$ we can use the smoothness of $q_1$ guaranteed by Lemma~\ref{lma:non-degenerate-smoothness} to see that $\hat s$ is also independent of $x$ and thus a global constant.
Therefore the first component of $q_1$ is smooth and non-vanishing.

The second component is simply
\begin{equation}
q_1(x,a,b)_2
=
\frac{P_1(x,a,b)_{12}}{q_1(x,a,b)_1}
,
\end{equation}
which is manifestly smooth.
The same argument works for the second eigenvector $q_2$ as well, so the eigenvectors are globally smooth, too. The functions $q_i$ are $0$-homogeneous in $\xi$ because $\Gamma$ is $2$-homogeneous.
\end{proof}

\subsection{Proofs of lemmas}

In the proofs in this subsection we may without notice shrink the domain to simplify the arguments.

\begin{proof}[Proof of Lemma~\ref{lma:smooth-singular-locus-on-bundle}]
As established in \cite{IKK2025}, a slowness surface is singular as a scheme at $(x_0,\xi_0)$ if and only the Christoffel matrix has a multiple eigenvalue at $(x_0,\xi_0)$. Hence $\mathcal{S}\subset\mathcal{S}':=\{(x,\xi):\mathrm{d}_\xi P(x,\xi)=0\}$, where $P$ denotes the slowness polynomial. Now on the other hand if $\mathrm{d}_\xi P(x,\xi)=0$, by homogeneity there is a constant $r$ such that $(x,r\xi)$ is on the slowness surface and hence the Christoffel matrix has a multiple eigenvalue at $(x,r\xi)$. Hence by homogeneity $\mathcal{S}'\subset\mathcal{S}$ and thus it suffices to characterize the set $\mathcal{S}'$. 

Since by assumption $(\mathrm{Hess}_\xi P)(x_0,\xi_0)$ is invertible, by the implicit function theorem we find a neighborhood $U_0$ of $x_0$ and a smooth function $\alpha:U_0\to\mathbb{R}^2$ such that
\begin{equation}
    (\mathrm{d}_\xi P)(x,\alpha(x))=0
\end{equation}
for all $x\in U_0$. Choose $\mathcal{B}$ such that $\pi_1(\mathcal{B})\subset U_0$. Then by homogeneity
\begin{equation}
    \mathcal{S}\cap\mathcal{B}=\mathcal{S}'\cap\mathcal{B}=\{(x,\mathbb{R}\alpha(x)):x\in\pi_1(\mathcal{B})\}.
\end{equation}
\end{proof}
\begin{remark}
    It is perhaps worth noting that we are not contradicting the uniqueness of $\alpha$ guaranteed by the implicit function theorem, since we find a single $\alpha$ for each function $r^3P$, $r\in\mathbb{R}$, not infinitely many functions $r\alpha$ for the single function $P$.
\end{remark}

\begin{proof}[Proof of Lemma~\ref{lma:global-existence+vertical-analyticity}]
For each $x$ and $\xi\in\mathbb{S}^1$ we can write
\begin{equation}
    \Gamma(x,\xi)=\Gamma(x,(\cos(\theta),\sin(\theta))=\Gamma(x,\theta).
\end{equation}
Since $\xi\mapsto\Gamma(x,\xi)$ is real analytic, we have that
\begin{equation}
    \theta\mapsto \Gamma(x,\theta)
\end{equation}
is real analytic on any open interval in $(0,2\pi]$. By choosing $\mathcal{B}$ small enough it suffices to consider $\theta\in I\subset (0,2\pi]$ for some open interval $I\subset (0,2\pi]$. The claim then follows by Rellich's theorem and homogeneity.
\end{proof}

\begin{proof}[Proof of Lemma~\ref{lma:non-degenerate-smoothness}]
Consider the function
\begin{equation}
    \Phi:(x,\xi,q,\lambda)\mapsto ((\Gamma-\lambda\mathrm{Id})q,(\vert q\vert^2-1)/2).
\end{equation}
Suppose $\Phi(x_0,\xi_0,q_0,\lambda_0)=0$ for some $(x,\xi)\in\mathcal{B}\setminus \mathcal{S}$. We have that
\begin{equation}
    (D_{(q,\lambda)}\Phi)(x_0,\xi_0,q_0,\lambda_0)=\begin{bmatrix}
        \Gamma(x_0,\xi_0)-\lambda_0\mathrm{Id} & q_0 \\
        -q^T_0 & 0
    \end{bmatrix}
\end{equation}
and
\begin{equation}
\begin{split}
\det((D_{(q,\lambda)}\Phi)(x_0,\xi_0,q_0,\lambda_0))
&=
\mathrm{tr}(\Gamma(x_0,\xi_0)-\lambda_0\mathrm{Id})\vert q_0\vert^2-\langle q_0,(\Gamma(x_0,\xi_0)-\lambda_0\mathrm{Id})q_0\rangle
\\
&=
\mathrm{tr}(\Gamma(x_0,\xi_0)-\lambda_0\mathrm{Id})\neq 0,
\end{split}
\end{equation}
since by the simplicity assumption the symmetric matrix $\Gamma(x_0,\xi_0)-\lambda_0\mathrm{Id}$ has exactly one nonzero eigenvalue. The claim follows by the implicit function theorem.
\end{proof}

\begin{proof}[Proof of Lemma~\ref{lma:Lipschitz}]
If $A(x)\in\R^{n\times n}$ is a symmetric matrix depending smoothly on $x\in\R^k$ and $\lambda(x)$ is a simple eigenvalue of $A(x)$, then the eigenvalue remains simple in a neighborhood of $x$ and
\begin{equation}
\frac{\partial}{\partial x_i}\lambda(x)
=
q_j(x)
q_k(x)
\frac{\partial}{\partial x_i}
A_{jk}(x)
,
\end{equation}
where $q$ is the corresponding unit eigenvector.
As $A(x)$ has uniformly bounded values and derivatives, these partials are uniformly bounded and thus the eigenvalue function is Lipschitz.
\end{proof}

\begin{proof}[Proof of Lemma~\ref{lma:ordinary node-fiber}]
By Lemma~\ref{lma:global-existence+vertical-analyticity} the two eigenvalues $\lambda_1(\xi),\lambda_2(\xi)$ of the Christoffel matrix $\Gamma(\xi)$ are real analytic functions of $\xi\neq0$.
We can write the slowness polynomial as $P_\bolda(\xi)=\det[\Gamma(\xi)-I]=[\lambda_1(\xi)-1][\lambda_2(\xi)-1]$.
At a singular point $\xi_0$ both eigenvalues equal one, and the Hessian of the product becomes simply
\begin{equation}
\label{eq:singular-Hessian}
D^2P_\bolda(\xi_0)
=
\der\lambda_1(\xi_0)\otimes\der\lambda_2(\xi_0)
+
\der\lambda_2(\xi_0)\otimes\der\lambda_1(\xi_0)
.
\end{equation}
By assumption $D^2P_\bolda(\xi_0)$ has rank $2$.
This can only happen in equation~\eqref{eq:singular-Hessian} if the differentials $\der\lambda_i(\xi_0)$ are linearly independent, and the claim follows.
\end{proof}

\begin{proof}[Proof of Lemma~\ref{lma:Hadamard}]
If $U$ does not intersect the set $\{x_1=0\}$, we may simply define $H(x)=x_1^{-k}f(x)$.
We may therefore assume that the sets do intersect.

Denoting $x=(x_1,x')$, we have by the fundamental theorem of calculus and the convexity of the set $U$ that
\begin{equation}
    f(x)=f(0,x')+\int_0^1\frac{\mathrm{d}}{\mathrm{d}t}f(tx_1,x')\ \mathrm{d}t=x_1\int_0^1(\partial_1f)(tx_1,x')\ \mathrm{d}t,
\end{equation}
since the function $f$ vanishes along $U\cap \{x_1=0\}$.
The claim follows for $k=1$ by the smoothness of the function $f$. 

Suppose then that the claim holds for any $1\leq k\leq K$. Then, if $f(x)\lesssim \vert x_1\vert^{K+1}$, we have
\begin{equation}
    \vert f(x)\vert =\vert x_1^KH(x)\vert=\vert x_1\vert^K\vert H(x)\vert\lesssim \vert x_1\vert^{K+1},
\end{equation}
giving us that
\begin{equation}
    \vert H(x)\vert\lesssim \vert x_1\vert.
\end{equation}
Applying the result in the case $k=1$ to the smooth function $H$ shows that
\begin{equation}
    H(x)=x_1\tilde{H}(x)
\end{equation}
for some smooth function $\tilde {H}$. The claim follows by induction.
\end{proof}

\appendix

\section{Relation to systems of involutive transverse type}
\label{app:involutive-fails}

Dencker defined systems of transverse type in~\cite{Dencker1995}. We do not use this terminology or its consequences, since in the case of ordinary nodes on slowness surfaces the elastic wave operator does not fit the definition; more precisely the involutivity condition given in~\cite{Dencker1995} holds if and only if the one form $\alpha$ given by Lemma~\ref{lma:smooth-singular-locus-on-bundle} is closed. To see when the involutivity condition is satisfied we use~\cite[Lemma A.26]{darwich2023propagation} which says that involutivity is equivalent to the pairwise vanishing of the Poisson brackets of all sets of defining functions with linearly independent differentials.

The singular locus of the slowness surface, denoted by $S$, is locally defined by $f_1=f_2=0$, where
$
    f_i(x,t,\xi,\tau):=\xi_i-\alpha_i(x)
$
for $i\in\{1,2\}$. Since
$
    (\mathrm{d}_\xi f_i)_j=\delta_{ij},
$
 the defining functions $f_i$ have linearly independent differentials. It is a straightforward computation that
\begin{equation*}
    \{f_1,f_2\}=\frac{\partial\alpha_2}{\partial x^1}-\frac{\partial\alpha_1}{\partial x^2}=0\iff \mathrm{d}\alpha=0.
\end{equation*}
Hence $S$
satisfies the involutivity condition if and only if the one form $\alpha$ is closed.

\section{Maxima code used for Lemma~\ref{lma:f''=0}}
\label{app:maxima}

Maxima code used in the proof of Lemma~\ref{lma:f''=0}:
\begin{verbatim}
/*general discriminant of 2x2 symmetric matrix*/
M:matrix([a,b],[b,c])$
I:matrix([1,0],[0,1])$
expand(determinant(M-L*I));
generaldiscriminant:expand((-c-a)^2-4*1*(a*c-b^2));

/*discriminant of Christoffel matrix*/
a:a11*p1^2 + 2*a13*p1*p2 + a33*p2^2$
b:a13*p1^2 + (a33 + a12)*p1*p2 + a23*p2^2$
c:a33*p1^2 + 2*a23*p1*p2 + a22*p2^2$
slownessdiscriminant:expand(ev(generaldiscriminant));

/*angular variations*/
f:ev(slownessdiscriminant,p1=cos(t),p2=sin(t));
d1f:diff(f,t)$
d2f:diff(f,t,2)$
f0:ev(f,t=0)$
d1f0:ev(d1f,t=0)$
d2f0:ev(d2f,t=0)$

conditions2:solve([f0,d1f0,d2f0],[a11,a12,a13,a22,a23,a33]);
length(conditions2);
imagpart(conditions2);
realpart(conditions2);
\end{verbatim}

\bibliographystyle{abbrv}
\bibliography{ref.bib}

@phdthesis{darwich2023propagation,
  author       = {Darwich, Rayhana}, 
  title        = {Propagation of polarization sets for systems of generalized transverse type and for systems of MHD type},
  school       = {University of G\"ottingen},
  year         = {2023},
  type = {Doctoral Dissertation}
}

@misc{IKK2025,
 author = {Ilmavirta, Joonas and Kirkkopelto, Pieti and Kykk{\"a}nen, Antti},
 title = {Horizontal and vertical regularity of elastic wave geometry},
 year = {2025},
 howpublished = {Preprint, {arXiv}:2511.16466 [math.{DG}] (2025)},
 url = {https://arxiv.org/abs/2511.16466},
 arXiv = {arXiv:2511.16466}
}

@misc{dHILV2023,
 author = {de Hoop, Maarten V. and Ilmavirta, Joonas and Lassas, Matti J. and V{\'a}rilly-Alvarado, Anthony},
 title = {Reconstruction of anisotropic stiffness tensors from partial data around one polarization},
 year = {2023},
 howpublished = {Preprint, {arXiv}:2307.03312 [math.{DG}] (2023)},
 url = {https://arxiv.org/abs/2307.03312},
 arXiv = {arXiv:2307.03312}
}

@article{Dencker1995,
author = {Dencker, Nils},
year = {1995},
month = {10},
pages = {249-279},
title = {The propagation of polarization for systems of transversal type},
volume = {33},
journal = {Ark. Mat.},
doi = {10.1007/BF02559709}
}

@article{OSSU2024,
 author = {Oksanen, Lauri and Salo, Mikko and Stefanov, Plamen and Uhlmann, Gunther},
 title = {Inverse problems for real principal type operators},
 fjournal = {American Journal of Mathematics},
 journal = {Am. J. Math.},
 issn = {0002-9327},
 volume = {146},
 number = {1},
 pages = {161--240},
 year = {2024},
 language = {English},
 doi = {10.1353/ajm.2024.a917541},
 zbMATH = {7791560},
 Zbl = {1531.35396}
}

@article{Dencker1982,
 author = {Dencker, Nils},
 title = {On the propagation of polarization sets for systems of real principal type},
 fjournal = {Journal of Functional Analysis},
 journal = {J. Funct. Anal.},
 issn = {0022-1236},
 volume = {46},
 pages = {351--372},
 year = {1982},
 language = {English},
 doi = {10.1016/0022-1236(82)90051-9},
 zbMATH = {3766740},
 Zbl = {0487.58028}
}

@article{BarceloEtAl2018,
  author  = {Barcel{\'o}, J. and Folch-Gabayet, M. and P{\'e}rez-Esteva, S. and Ruiz, A. and Vilela, M.},
  title   = {Uniqueness for inverse elastic medium problems},
  journal = {SIAM J. Math. Anal.},
  volume  = {50},
  pages   = {3939--3962},
  year    = {2018}
}

@article{BerettaFranciniVessella2014,
  author  = {Beretta, E. and Francini, E. and Vessella, S.},
  title   = {Uniqueness and {Lipschitz} stability for the identification of {Lam{\'e}} parameters from boundary measurements},
  journal = {Inverse Probl. Imaging},
  volume  = {8},
  pages   = {611--644},
  year    = {2014}
}

@article{BraamDuistermaat1993,
  author  = {Braam, P. J. and Duistermaat, J. J.},
  title   = {Normal forms of real symmetric systems with multiplicity},
  journal = {Indag. Math. (N.S.)},
  volume  = {4},
  number  = {4},
  pages   = {407--421},
  year    = {1993}
}

@article{ColinDeVerdiere2003,
 author = {Colin de Verdi{\`e}re, Yves},
 title = {The level crossing problem in semi-classical analysis. {I}: {The} symmetric class.},
 fjournal = {Annales de l'Institut Fourier},
 journal = {Ann. Inst. Fourier},
 issn = {0373-0956},
 volume = {53},
 number = {4},
 pages = {1023--1054},
 year = {2003},
 language = {English},
 doi = {10.5802/aif.1973},
 url = {https://eudml.org/doc/116061},
 zbMATH = {2014671},
 Zbl = {1113.35151}
}

@article{ColinDeVerdiere2004,
  author  = {Colin de Verdi{\`e}re, Y.},
  title   = {The level crossing problem in semi-classical analysis. {II}. The {Hermitian} case},
  journal = {Ann. Inst. Fourier},
  volume  = {54},
  number  = {5},
  pages   = {1423--1441},
  year    = {2004}
}

@article{DeHoopIlmavirtaLassas2025,
  author  = {de Hoop, M. V. and Ilmavirta, J. and Lassas, M.},
  title   = {Reconstruction along a geodesic from sphere data in {Finsler} geometry and anisotropic elasticity},
  journal = {J. Math. Pures Appl.},
  volume  = {196},
  pages   = {103688},
  year    = {2025}
}

@article{DeHoopIlmavirtaLassasSaksala2021,
  author  = {de Hoop, M. V. and Ilmavirta, J. and Lassas, M. and Saksala, T.},
  title   = {A foliated and reversible {Finsler} manifold is determined by its broken scattering relation},
  journal = {Pure Appl. Anal.},
  volume  = {3},
  number  = {4},
  pages   = {789--811},
  year    = {2021}
}

@article{DeHoopIlmavirtaLassasSaksala2023,
  author  = {de Hoop, M. V. and Ilmavirta, J. and Lassas, M. and Saksala, T.},
  title   = {Determination of a compact {Finsler} manifold from its boundary distance map and an inverse problem in elasticity},
  journal = {Comm. Anal. Geom.},
  volume  = {31},
  number  = {7},
  pages   = {1693--1747},
  year    = {2023}
}

@article{DeHoopUhlmannVasy2020,
  author  = {de Hoop, M. V. and Uhlmann, G. and Vasy, A.},
  title   = {Recovery of material parameters in transversely isotropic media},
  journal = {Arch. Rational Mech. Anal.},
  volume  = {235},
  number  = {1},
  pages   = {141--165},
  year    = {2020}
}

@article{Dencker1988,
  author  = {Dencker, N.},
  title   = {On the propagation of polarization in conical refraction},
  journal = {Duke Math. J.},
  volume  = {57},
  number  = {1},
  pages   = {85--134},
  year    = {1988}
}

@article{EskinRalston2002,
  author  = {Eskin, G. and Ralston, J.},
  title   = {On the inverse boundary value problem for linear isotropic elasticity},
  journal = {Inverse Problems},
  volume  = {18},
  number  = {3},
  pages   = {907--921},
  year    = {2002}
}

@article{HelbigCarcione2009,
  author  = {Helbig, K. and Carcione, J.},
  title   = {Anomalous polarization in anisotropic media},
  journal = {Eur. J. Mech. A Solids},
  volume  = {28},
  pages   = {704--711},
  year    = {2009}
}

@misc{IlmavirtaMonkkonen2022,
  author        = {Ilmavirta, J. and M{\"o}nkk{\"o}nen, K.},
  title         = {The geodesic ray transform on spherically symmetric reversible {Finsler} manifolds},
  year          = {2022},
  eprint        = {2203.16886},
  archivePrefix = {arXiv},
  primaryClass  = {math.DG}
}

@article{ImanuvilovUhlmannYamamoto2012,
  author  = {Imanuvilov, O. and Uhlmann, G. and Yamamoto, M.},
  title   = {On reconstruction of {Lam{\'e}} parameters from partial {Cauchy} data in three dimensions},
  journal = {Inverse Problems},
  volume  = {28},
  pages   = {125002},
  year    = {2012}
}

@article{MazzucatoRachele2006,
  author  = {Mazzucato, A. and Rachele, L.},
  title   = {Partial uniqueness and obstruction to uniqueness in inverse problems for anisotropic elastic media},
  journal = {J. Elasticity},
  volume  = {83},
  number  = {3},
  pages   = {205--245},
  year    = {2006}
}

@article{MazzucatoRachele2007,
  author  = {Mazzucato, A. and Rachele, L.},
  title   = {On uniqueness in the inverse problem for transversely isotropic elastic media with a disjoint wave mode},
  journal = {Wave Motion},
  volume  = {44},
  number  = {7},
  pages   = {605--625},
  year    = {2007}
}

@article{MazzucatoRachele2008,
  author  = {Mazzucato, A. and Rachele, L.},
  title   = {On transversely isotropic elastic media with ellipsoidal slowness surfaces},
  journal = {Math. Mech. Solids},
  volume  = {13},
  number  = {7},
  pages   = {611--638},
  year    = {2008}
}

@article{MelroseUhlmann1979,
  author  = {Melrose, R. and Uhlmann, G.},
  title   = {Lagrangian intersection and the {Cauchy} problem},
  journal = {Comm. Pure Appl. Math.},
  volume  = {32},
  number  = {4},
  pages   = {483--519},
  year    = {1979}
}

@article{NakamuraTanumaUhlmann1999,
  author  = {Nakamura, G. and Tanuma, K. and Uhlmann, G.},
  title   = {Layer stripping for a transversely isotropic elastic medium},
  journal = {SIAM J. Appl. Math.},
  volume  = {59},
  number  = {5},
  pages   = {1879--1891},
  year    = {1999}
}

@article{NakamuraUhlmann1994,
  author  = {Nakamura, G. and Uhlmann, G.},
  title   = {Global uniqueness for an inverse boundary problem arising in elasticity},
  journal = {Invent. Math.},
  volume  = {118},
  number  = {3},
  pages   = {457--474},
  year    = {1994}
}

@book{PaternainSaloUhlmann2023,
  author    = {Paternain, G. and Salo, M. and Uhlmann, G.},
  title     = {Geometric inverse problems: with emphasis on two dimensions},
  publisher = {Cambridge University Press},
  volume    = {204},
  year      = {2023}
}

@article{SacksYakhno1998,
  author  = {Sacks, P. E. and Yakhno, V. G.},
  title   = {The inverse problem for a layered anisotropic half space},
  journal = {J. Math. Anal. Appl.},
  volume  = {228},
  number  = {2},
  pages   = {377--398},
  year    = {1998}
}

@article{Uhlmann1982,
  author  = {Uhlmann, G.},
  title   = {Light intensity distribution in conical refraction},
  journal = {Comm. Pure Appl. Math.},
  volume  = {35},
  pages   = {69--80},
  year    = {1982}
}

@article{UhlmannVasy2016,
  author  = {Uhlmann, G. and Vasy, A.},
  title   = {The inverse problem for the local geodesic ray transform},
  journal = {Invent. Math.},
  volume  = {205},
  number  = {1},
  pages   = {83--120},
  year    = {2016}
}

@article{Yedlin1980,
  author  = {Yedlin, M. J.},
  title   = {The wave front in a homogeneous anisotropic medium},
  journal = {Bull. Seismol. Soc. Am.},
  volume  = {70},
  number  = {6},
  pages   = {2097--2102},
  year    = {1980}
}

@misc{zou2021,
 author = {Yuzhou Zou},
 title = {Microlocal methods for the elastic travel time tomography problem for transversely isotropic media},
 year = {2021},
 howpublished = {Preprint, {arXiv}:2112.14455 [math.{AP}]},
 url = {https://arxiv.org/abs/2112.14455},
 arXiv = {arXiv:2112.14455}
}

@book{Hintz2025,
 author = {Hintz, Peter},
 title = {An introduction to microlocal analysis},
 fseries = {Graduate Texts in Mathematics},
 series = {Grad. Texts Math.},
 issn = {0072-5285},
 volume = {304},
 isbn = {978-3-031-90705-0; 978-3-031-90708-1; 978-3-031-90706-7},
 year = {2025},
 publisher = {Cham: Springer},
 language = {English},
 doi = {10.1007/978-3-031-90706-7},
 zbMATH = {8014458}
}

@book{CLO2025,
 author = {Cox, David A. and Little, John and O'Shea, Donal},
 title = {Ideals, varieties, and algorithms. {An} introduction to computational algebraic geometry and commutative algebra},
 edition = {5th edition},
 fseries = {Undergraduate Texts in Mathematics},
 series = {Undergraduate Texts Math.},
 issn = {0172-6056},
 isbn = {978-3-031-91840-7; 978-3-031-91843-8; 978-3-031-91841-4},
 year = {2025},
 publisher = {Cham: Springer},
 language = {English},
 doi = {10.1007/978-3-031-91841-4},
 zbMATH = {8074064}
}

@book{SK2000,
 author = {Smith, Karen E. and Kahanp{\"a}{\"a}, Lauri and Kek{\"a}l{\"a}inen, Pekka and Traves, William},
 title = {An invitation to algebraic geometry},
 fseries = {Universitext},
 series = {Universitext},
 issn = {0172-5939},
 isbn = {0-387-98980-3},
 year = {2000},
 publisher = {New York, NY: Springer},
 language = {English},
 zbMATH = {1515219},
 Zbl = {0962.14001}
}

@book{CA2000,
 author = {Casas-Alvero, Eduardo},
 title = {Singularities of plane curves},
 fseries = {London Mathematical Society Lecture Note Series},
 series = {Lond. Math. Soc. Lect. Note Ser.},
 issn = {0076-0552},
 volume = {276},
 isbn = {0-521-78959-1},
 year = {2000},
 publisher = {Cambridge: Cambridge University Press},
 language = {English},
 zbMATH = {1497487},
 Zbl = {0967.14018}
}

\end{document}